\documentclass[11pt]{article}  
\usepackage[english]{babel}
\usepackage[toc,page]{appendix}
\usepackage{bbm}
\usepackage{fullpage}
\usepackage[top=1in, bottom=1in, left=1in, right=1in]{geometry}
\usepackage{etaremune}
\usepackage{enumitem}

\usepackage{amssymb,amsmath,amsthm}

\usepackage{hyperref}

\usepackage[showonlyrefs]{mathtools}

\DeclareSymbolFont{rsfs}{U}{rsfs}{m}{n}
\DeclareSymbolFontAlphabet{\mathscrsfs}{rsfs}
\usepackage{mathrsfs}

\usepackage{url}

\hypersetup{
  colorlinks   = true, 
  urlcolor     = blue, 
  linkcolor    = blue, 
  citecolor   = red 
}

\usepackage[labelformat=empty]{caption}

\newtheorem{theorem}{Theorem}[section]
\newtheorem{lemma}[theorem]{Lemma}

\newtheorem{proposition}[theorem]{Proposition}

\theoremstyle{definition}
\newtheorem{definition}{Definition}
\newtheorem{remark}[theorem]{Remark}

\numberwithin{equation}{section}

\usepackage{times}

\def\reals{{\mathbb R}}
\def\R{{\mathbb R}}
\def\Z{{\mathbb Z}}
\newcommand{\bea}{\begin{eqnarray}}
\newcommand{\eea}{\end{eqnarray}}
\newcommand{\<}{\langle}
\renewcommand{\>}{\rangle}
\def\bsigma{{\boldsymbol{\sigma}}}
\def\bbE{{\mathbb{E}}}

\def\pl{\mbox{\rm\tiny pl}}
\def\plsmall{\mbox{\rm\small pl}}
\def\rd{\mbox{\rm\tiny rd}}
\def\cP{{\mathcal{P}}}
\def\Par{{\sf P}}
\def\RS{\mbox{\tiny\rm RS}}
\def\de{{\rm d}}
\def\la{\langle}
\def\ra{\rangle}

\def\Is{\mathrm{Is}}

\newcommand{\Sph}{\mathbb{S}}
\newcommand{\ubq}{\bar{q}}
\newcommand{\lbq}{\underline{q}}
\def\bsig{{\boldsymbol {\sigma}}}
\def\bG{{\boldsymbol G}}
\def\P{\mathbb{P}}
\DeclareMathOperator*{\E}{\bbE}
\def\supp{{\rm supp}}
\def\eps{\varepsilon}
\def\one{{\mathbf{1}}}
\def\cL{{\cal L}}
\def\cont{{\mathsf{cont}}}
\def\dis{{\mathsf{dis}}}
\DeclareMathOperator*{\argmax}{\arg\max}

\author{
Antonio Auffinger \thanks{Northwestern University, Department of Mathematics. Evanston, IL, USA. Email: \texttt{auffing@math.northwestern.edu}}
\and
Ahmed El Alaoui \thanks{Cornell University, Department of Statistics and Data Science. Ithaca, NY, USA. Email: \texttt{elalaoui@cornell.edu}}
\and
Mark Sellke\thanks{Harvard University, Department of Statistics. Cambridge, MA, USA. Email: \texttt{msellke@fas.harvard.edu}}
}

\title{On the Discontinuous Breaking of Replica Symmetry and Shattering in Mean-Field Spin Glasses}

\date{}

\begin{document}

\maketitle

\begin{abstract}
\noindent
We show that in mean-field spin glasses, a discontinuous breaking of replica symmetry at the critical inverse temperature $\beta_c$ implies the existence of an intermediate shattered phase.  
This confirms a prediction from physics regarding the nature of random first order phase transitions. On the other hand, we give an example of a spherical spin glass which exhibits shattering, yet the transition is continuous at $\beta_c$.  
\end{abstract}

\section{Introduction and main result}

Mean-field spin glasses are understood to exhibit multiple phase transitions as the temperature is lowered, including the famous replica symmetry breaking phenomenon.
The onset of replica symmetry breaking may take one of two forms: the order parameter becomes non-trivial either \emph{continuously} or \emph{discontinuously}.
The latter is often referred to as a \emph{random first-order phase transition}, and has long been believed by physicists to be associated with the existence of an intermediate \emph{dynamical} or \emph{shattering} transition (see the discussion in e.g. \cite{mezard1996tentative,franz2013note,kirkpatrick2014universal}).
We consider both Ising and spherical glasses and prove that a discontinuous phase transition implies the existence of a shattered phase, but that the converse does not hold in spherical spin glasses.

 Let $(g_{i_1,\cdots,i_k})_{1 \le i_1,\cdots,i_k \le N, k \ge 2}$ be a sequence of i.i.d.\ $N(0,1)$ random variables and consider the random Hamiltonian 
\begin{align}\label{eq:hamiltonian}
H_N(\bsigma) =  \sum_{k=2}^{\infty} \frac{\gamma_k}{N^{(k-1)/2}}  \sum_{1\le i_1 ,\cdots , i_k\le N} g_{i_1,\cdots,i_k}\sigma_{i_1}\cdots\sigma_{i_k} \,,
\end{align}
defined for $\bsigma \in \Sigma_N$ being either the Euclidean sphere on radius $\sqrt{N}$: $\Sigma_N = \Sph^{N-1}(\sqrt{N}) =\{\bsigma \in \reals^N : \|\bsigma\| \le \sqrt{N}\}$, or the binary hypercube $\Sigma_N = \{-1,+1\}^N$.
We let 
\[\xi(x) := \sum_{k=2}^{\infty} \gamma_k^2 x^k\]  
be the mixture function of $H_N$ and we assume that $\xi(1)=1$ and $\xi(1+\alpha) <\infty$ for some $\alpha>0$ so that $H_N$ is almost surely well defined.  

 We define the Gibbs measure at inverse temperature $\beta$:
\begin{align}\label{eq:Gibbs}
\de \mu_{\beta}(\bsigma) \propto e^{\beta H_N(\bsigma)} \de \mu_{0}(\bsigma)\,,
\end{align}
where $\mu_0$ is the uniform measure on $\Sigma_N$ (spherical or binary).

Define the limiting free energy of the model as
 \begin{align}
F(\beta) &= \lim_{N\to \infty} \frac{1}{N} \E \log \int e^{\beta H_N(\bsigma)} \de\mu_0( \bsigma)\,,
\end{align}
where it is known that the limit exists and is given by a Parisi formula for all $\beta \ge 0$ \cite{parisi1979infinite,talagrand2006parisi,talagrand2006spherical,panchenko2013parisi,chen2013aizenman}.
Let $\beta_c$ be the replica symmetry breaking inverse temperature defined as 
\begin{align}\label{eq:betac}
\beta_c = \sup\{\beta \ge 0 : F(\beta) = \beta^2/2\}\,.
\end{align}
It is known that $\beta_c<\infty$ and that $F(\beta)=\beta^2/2$ for all $0\leq\beta\leq\beta_c$.

We are interested in the nature of the phase transition at $\beta_c$ and its impact on the structure of the Gibbs measure $\mu_{\beta}$ at higher temperatures $\beta < \beta_c$. It is known that the two-replica overlap distribution is a delta mass at $0$ for all $\beta \le \beta_c$. Informally, we say that the phase transition at $\beta_c$ happens \emph{continuously} if the point mass at $0$ splits into continuously moving pieces as the inverse temperature $\beta$ is increased past $\beta_c$, i.e., the support of the measure is continuous in $\beta$ (as a compact subset of $[0,1]$ metrized by the Hausdorff distance). Conversely, we say that the phase transition happens \emph{discontinuously} if a piece of the support emerges far away from 0 as soon as $\beta>\beta_c$.  
More formally, we let $\cP([0,1])$ denote the set of probability measures on the interval $[0,1]$, and for $\zeta \in \cP([0,1])$ we define
\begin{align}
q_{\max}(\zeta)
=
\sup\{q\in \supp(\zeta)\}\,,
\end{align}
to be the right-end point of the support of $\zeta$.
Let $\zeta_{\beta}=\zeta_{\beta}(\xi)$ be the Parisi measure of the model, i.e. the minimizer of the Parisi formula as recalled in Section~\ref{sec:preliminaries}.

\begin{definition}
\label{def:discontinuous}
    We say $\xi$ admits a \textbf{continuous} RSB transition if 
    \[
    \lim_{\beta\downarrow \beta_c}
    q_{\max}(\zeta_{\beta})
    =
    0\, .
    \]
    We say $\xi$ admits a \textbf{discontinuous} RSB transition if
    \[
    \lim_{\beta\downarrow \beta_c}
    q_{\max}(\zeta_{\beta})
    >0\, .
    \]
\end{definition}

As our main result we show that a discontinuous transition implies the existence of a \emph{shattered} phase immediately below $\beta_c$. 
For $q \in [0,1]$ and $\beta\leq \beta_c$ we define the Franz-Parisi potential
\begin{equation}
\label{eq:FP-def}
F_{\beta}(q) := 
\lim_{\eps \downarrow 0} \lim_{N \to \infty} \frac{1}{N} 
\E \log \int \one_{| \langle \bsigma, \bsigma_0\rangle/N - q| \le \eps} \, e^{\beta H_N(\bsigma)} \de\mu_0( \bsigma)\, ,
\end{equation}
where the expectation $\bbE[\,\cdot\,]$ 
is taken with respect to the joint distribution of the disorder, i.e., the random variables $(g_{i_1,\cdots,i_k})$, and $\bsigma_0 \sim \mu_{\beta}$. 
In the spherical case, the above limit was shown to exist and computed in \cite{alaoui2023shattering} for all $\beta\leq \beta_c$; see also \cite{dembo2024disordered} for low-temperature generic models with finite replica symmetry breaking, where the proof tools make use of spherical symmetry. Although the exact formula will not be needed here, we also compute the FP potential~\eqref{eq:FP-def} in the Ising case, for $\beta<\beta_c$, using a more general convex duality argument in Section~\ref{sec:FP}.
We are now in position to formally define the notion of shattering. 
\begin{definition}\label{def:shattering}
We say that $\xi$ exhibits \textbf{shattering} at inverse temperature $\beta$ if
\begin{enumerate}
\item $F_{\beta}(0) = F(\beta)$.
\item $F_{\beta}(q) < F(\beta)$ for all $q \in (0,1)$.
\item There exists an interval $ [q_1 , q_2] \subset (0,1)$ on which $F_{\beta}$ is strictly increasing. 
\end{enumerate}
\end{definition}

In the spherical case the first two conditions are satisfied for all $\beta < \beta_c$ and the third condition was shown to hold for the pure $p$-spin model for all $\beta \in (C,\beta_c)$ where $C$ is an absolute constant for $p$ large~\cite{alaoui2023shattering} while $\beta_c(p)=(1+o_p(1))\cdot\sqrt{\log p}$. This latter condition heuristically signals a ``clumping" of the mass of the Gibbs measure $\mu_{\beta}$ around its typical points. It was shown in the same paper that under the additional technical assumption that $q_1$ is close enough to 1, one can construct a \emph{shattering decomposition} for $\mu_{\beta}$: a collection of well separated connected subsets of the sphere, each with exponentially small Gibbs mass, which together carry all but an exponentially small fraction of the Gibbs mass. See also~\cite{alaoui2024near} for a different approach to shattering yielding an improved lower bound on $\beta$.   
We show in Lemma~\ref{lem:FPglobal} that the first two conditions also hold in the Ising case for all $\beta<\beta_c$. The third condition on the other hand will be implied by a discontinuous breaking of replica symmetry.         

In addition to $\beta_c$ defined in Eq.~\eqref{eq:betac} we define the  inverse temperature
\begin{align}
\beta_{\cont} &:= \xi''(0)^{-1/2}\,. \label{eq:beta_cont}
\end{align}

One interpretation of $\beta_{\cont}$ is that it is the largest inverse temperature such that $q=0$ is a local maximum of the FP potential (see Lemma~\ref{lem:FPbound}). 
We now state our main result:
\begin{theorem}
\label{thm:main}
Suppose $\xi$ admits a discontinuous RSB transition and $\beta_c < \beta_\cont$. Then there exists $\delta>0$ such that $\xi$ exhibits shattering for all $\beta \in [\beta_c - \delta,\beta_c)$. 
\end{theorem}

The additional condition $\beta_c < \beta_\cont$ in the above statement is a genericity condition on the mixture $\xi$: one always has $\beta_c \le \beta_\cont$ (see Proposition~\ref{prop:beta-cont-bound} ), and in the spherical case we show in Lemma~\ref{lem:beta_order} that the inequality is strict under small perturbations of $\xi$ whenever the RSB transition is discontinuous.

We also mention that in the spherical case, shattering at inverse temperature $\beta$ is conjectured in the physics literature~\cite{crisanti1993spherical} to hold exactly when there exists $q \in (0,1)$ such that 
\[
\beta^2\xi'(q)(1-q) > q\,.
\]
We expect this condition to be equivalent to Definition~\ref{def:shattering}, but this is currently not known (see \cite[Corollary 2.6]{dembo2025dynamics} for important recent progress from the viewpoint of stationary Langevin dynamics). For completeness we also show in Section~\ref{sec:spherical} that this alternative criterion is implied by a discontinuous RSB as per Definition~\ref{def:discontinuous}. On the other hand we provide an example of a mixture $\xi$ which admits a continuous RSB transition yet exhibits shattering, showing that the implication of Theorem~\ref{thm:main} cannot be reversed in general.

\begin{remark}
\label{rem:shattering}
As a further remark, shattering as in Definition~\ref{def:shattering} implies exponentially slow mixing of natural dynamics initialized from equilibrium. This is simply because for a typical point $\bsig_0\sim\mu_{\beta}$ the set $\{\bsig:\la\bsig,\bsig_0\ra/N\geq q\}$ yields a bottleneck for any $q \in (q_1,q_2)$. 
(We note that \cite[Corollary 2.8]{alaoui2023shattering} gives a similar deduction of slow mixing, based on a more refined shattering decomposition that we do not focus on here.)
We sketch a crude argument: let $q_1< \lbq < \ubq < q_2$ and $\alpha:[-1,1]\to [0,1]$ be a Lipschitz increasing function with $\alpha(\lbq)=0$ and $\alpha(\ubq)=1$.
We argue that the test function 
\[
f(\bsig)
=
\alpha(\la\bsig,\bsig_0\ra/N)
\]
witnesses an exponentially small spectral gap for the dynamics: 
\begin{equation}
\label{eq:rayleigh-quotient}
\la f,(-\cL_{\beta}) f\ra\leq e^{-cN} \|f-\bbE_{\mu_{\beta}}[f]\|_{L^2(\mu_\beta)}^2
\end{equation}
for some $c>0$ with probability tending to $1$ as $N \to \infty$, where $\cL_{\beta}$ is the infinitesimal generator for either Langevin or Glauber dynamics. We refer to~\cite{gheissari2019spectral} and~\cite{eldan2021spectral} for a definition of $\cL_{\beta}$ in the spherical and Ising settings respectively. 
Indeed let us define $\bar{F}_{\beta}(q) = F_{\beta}(q) - F(\beta)$. 
In the argument to follow, all displayed inequalities hold on events of high probability.  
We have by the second item  of Definition~\ref{def:shattering} that 
\begin{equation*}
\bbE_{\mu_{\beta}}[f] \leq \exp\big(N \sup_{q\in [\lbq,1]} \bar{F}_{\beta}(q) +o(N)\big)\,,~~~\mbox{and}~~~
\bbE_{\mu_{\beta}}[f^2] \ge \exp\big(N \sup_{q\in [\ubq,1]} \bar{F}_{\beta}(q) -o(N)\big)\,.
\end{equation*}
To obtain the above we used the bounds $f(\bsig) \le \one\{\la\bsig,\bsig_0\ra/N \ge \lbq\}$ and $f(\bsig)^2\ge \one\{\la\bsig,\bsig_0\ra/N \ge \ubq\}$. 
Since $\bar{F}_{\beta}$ is strictly increasing on $(q_1,q_2)$, the above two suprema are equal. Moreover, since $\bar{F}_{\beta}(q)<0$ for $q\neq 0$, we have  
\[ \|f-\bbE_{\mu_{\beta}}[f]\|_{L^2(\mu_\beta)}^2 \ge \exp\big(N \sup_{q\in [\ubq,1]} F_{\beta}(q) - o(N)\big)\geq \exp(N F_{\beta}(q_2)- o(N))\,.\]

As for the left-hand side of \eqref{eq:rayleigh-quotient}, since $\alpha$ is constant on $[0,\lbq]$ and $[\ubq,1]$, 
\[\la f,(-\cL_{\beta}) f\ra \le \exp\big(N \sup_{q\in [\lbq,\ubq]}
\bar{F}_{\beta}(q) +o(N)\big) = \exp\big(N
\bar{F}_{\beta}(\ubq) +o(N)\big)\,.\]
The conclusion~\eqref{eq:rayleigh-quotient} follows since $\bar{F}_{\beta}(\ubq) < \bar{F}_{\beta}(q_2)$: the spectral gap of Langevin or Glauber dynamics is exponential in $N$ whenever $\xi$ exhibits shattering.
\end{remark}

\section{Technical Preliminaries}
\label{sec:preliminaries}

\paragraph{The Parisi formula}
We recall the Parisi functional.
In the spherical case, it is the strictly convex lower semicontinuous (in the weak$^*$ topology) functional 
$\Par:\cP([0,1))\to \reals$
on the space of probability measures $\zeta$ with support contained inside $[0,1]$:
\begin{align}
\label{eq:CrisantiSommers}
\Par^{\Sph}_{\beta}(\zeta;\xi) 
&=
\frac{1}{2}\left(\int_{0}^1 \beta^2(\xi'(t)+h^2)
\zeta([0,t])
\de t 
+
\int_0^{\hat{q}} \frac{\de t}{\phi_{\zeta}(t)}  +\log(1-\hat{q})\right)\, ,
\\
\nonumber
\phi_{\zeta}(t) & = \int_{t}^1\zeta([0,s])\de s\, ,~~~ \hat{q} = \inf\{s\in [0,1) \, :\, \zeta([0,s])=1\}{\,<1}\, .
\end{align}
In the Ising case, for $\zeta\in\cP([0,1])$ one defines $\Phi:[0,1]\times\reals\to\reals$ to solve the Parisi PDE
\begin{equation}
\label{eq:Parisi-PDE}
\begin{aligned}
\partial_t \Phi_{\zeta}(t,x)
&=
-\frac{\beta^2 \xi''(t)}{2}\Big(
\partial_{xx}\Phi_{\zeta}(t,x)+\zeta([0,s])(\partial_x \Phi_{\zeta}(t,x))^2
\Big),
\\
\Phi_{\zeta}(1,x)
&=
\log\cosh(x).
\end{aligned}
\end{equation}
Then the Ising Parisi functional is
\[
\Par^{\Is}_{\beta}(\zeta;\xi,h) 
=
\Phi_{\zeta}(0,h)-\frac{\beta^2}{2}\int_0^1 \zeta([0,t])t\xi''(t)\de t.
\]

\begin{proposition}[\cite{talagrand2006parisi-b,auffinger2015parisi}]
\label{prop:parisi-formula}
In both the Ising and spherical settings, one has 
\[
F(\beta;\xi,h)
=
\lim_{N\to\infty}
\bbE
\log 
\int 
e^{\beta H_N(\bsigma)+h\la \one,\sigma\ra}
\de\mu_0(\sigma) = 
\inf_{\zeta\in\cP([0,1])}
\Par_{\beta}(\zeta;\xi).
\]
Moreover $\Par_{\beta}(\cdot;\xi)$ is strictly convex and the minimum is uniquely attained at some $\zeta_{\beta}(\xi)\in\cP([0,1])$.
Finally $\beta\in [0,\beta_c]$ if and only if $\zeta_{\beta}=\delta_0$, the atom at $0$.
\end{proposition}

The Parisi formula can be differentiated explicitly; see e.g. \cite{talagrand2006parisi,panchenko2008differentiability}:

\begin{proposition}
\label{prop:free-energy-deriv}
The derivative of $F(\beta;\xi)$ with respect to any $\gamma_p$ exists and is given by
\[
\frac{\partial F(\beta;\xi)}{\partial\gamma_p}
=
\beta^2 \gamma_p 
\Big(1-\int q^p  \de \zeta_{\beta}(q)\Big)\,,
\]
and moreover, for each $p$ such that $\gamma_p \neq 0$ we have
\[ \lim_{N \to \infty} \E \mu_{\beta}^{\otimes 2} \big( (\langle \bsigma_1 , \bsigma_2 \rangle/N)^p \big) = \int q^p  \de \zeta_{\beta}(q)\,.\]
Furthermore, $\zeta_{\beta}$ is continuous in $\xi$.
\end{proposition}


\paragraph{The planted model and contiguity} It will be useful to define a \emph{planted} model under which several computations become tractable. Let $\bsigma \sim \mu_0$ and for all $k \ge 2$, $1 \le i_1,\cdots,i_k \le N$,
\begin{equation}\label{eq:pl}
g_{i_1,\cdots,i_k} = \frac{\beta \gamma_k}{N^{(k-1)/2}} \sigma_{i_1}\cdots \sigma_{i_k} + \tilde{g}_{i_1,\cdots,i_k}\,, 
\end{equation}
where $\tilde{g}_{i_1,\cdots,i_k} \sim N(0,1)$ independently of each other and of $\bsigma$. We denote by $\bG$ the collection of random variables $(g_{i_1,\cdots,i_k})_{1 \le i_1,\cdots, i_k \le N, k \ge 2}$ and let $\nu_{\pl}$ be the joint distribution of the pair $(\bG,\bsigma)$ generated this way. 

On the other hand let $\nu_{\rd}$ be the joint law of $(\bG,\bsigma)$ where $g_{i_1,\cdots,i_k} \sim N(0,1)$ i.i.d.\ and $\bsigma \sim \mu_{\beta}$ conditionally on $\bG$. By a straightforward computation their likelihood ratio is given by a normalized version of the partition function of the Gibbs measure $\mu_{\beta}$:
\begin{equation}\label{eq:lr} 
\frac{\de \nu_{\pl}}{\de \nu_{\rd}} (\bG,\bsigma) = \int e^{\beta H_N(\bsigma') - \beta^2 N \xi(1)/2} \de \mu_0(\bsigma')\,. 
\end{equation}%
In particular the above does not depend on $\bsigma$ and will henceforth be denoted by $L_N(\bG)$, and $H_N$ is defined as in Eq.~\eqref{eq:hamiltonian} with disorder coefficients given by $\bG$. 
Since $\bG \mapsto \log L_N(G)$ is $\beta \sqrt{N}$-Lipschitz function and $\frac{1}{N} \E\log L_N(\bG) \to 0$ for $\bG \sim \nu_{\rd}$ and all $\beta \le \beta_c$ by definition of $\beta_c$ Eq.~\eqref{eq:betac}, we have by Gaussian concentration of Lipschitz functions~\cite[Theorem 5.6]{boucheron2013concentration},
\begin{equation} \label{eq:concent}
\P\big( |\log L_N(\bG)| \ge t N\big) \le e^{-cN} \,,
\end{equation}
where $c = c(t,\beta)>0$ for all $t>0$, $\beta \le \beta_c$ and all $N$ sufficiently large. 
This implies the useful fact that any event which is exponentially unlikely under $\nu_{\pl}$ is also exponentially unlikely under $\nu_{\rd}$ (and vice-versa but this won't be used):
\begin{lemma} \label{lem:contig}
Let $\beta\le \beta_c$.     For any sequence of events $(E_N)$ defined on the common probability space of $\nu_{\rd}$ and $\nu_{\pl}$, if $\nu_{\pl}(E_N) \le e^{-cN}$ for $c>0$ and $N$ sufficiently large, then $\nu_{\rd}(E_N) \le e^{-c'N}$ for some $c'>0$ and $N$ sufficiently large.
\end{lemma}
The above property is referred to in \cite{alaoui2023shattering} as ``contiguity at exponential scale"; see Lemma 3.5 therein for a short proof using \eqref{eq:concent}.
%


\section{Proof of Theorem~\ref{thm:main}}
\label{sec:proofmain}
 We start by establishing a few facts about the local and global behavior of $F_{\beta}$. 
\begin{lemma}
\label{lem:FPbound}
For all $\beta < \beta_{\cont} = \xi''(0)^{-1/2}$, there is $\hat{q}>0$ with $F_{\beta}(q) < F(\beta)$ for all $q \in (0,\hat{q})$.
\end{lemma} 

\begin{proof}
The argument relies on a first computation under the planted model $\nu_{\pl}$, as used in \cite{alaoui2022sampling,alaoui2023shattering}. For $\bsigma \in \Sigma_N$ we let $B_{q,\eps}(\bsigma) = \{\bsigma' \in \Sigma_N : |\langle \bsigma , \bsigma'\rangle/N - q|\le \eps\}$. We consider the restricted partition function
\[Z_N = Z_N(\bsigma_0,q,\eps) := \int_{B_{q,\eps}(\bsigma_0)}  e^{\beta H_{N}(\bsigma)} \de \mu_0(\bsigma)\,. \]
Taking expectations under the planted model $(\bG,\bsigma_0) \sim \nu_{\pl}$ defined in Eq.~\eqref{eq:pl} and writing $\tilde{H}_N$ for the Hamiltonian with disorder random variables $(\tilde{g}_{i_1,\cdots,i_k})$ we have 
\begin{align}
\frac{1}{N} \log \bbE_{\plsmall}[Z_N] &= \frac{1}{N} \log \int_{B_{q,\eps}(\bsigma_0)} \, \E e^{\beta \tilde{H}_{N}(\bsigma) + N\beta^2 \xi(\<\bsigma , \bsigma_0\>/N)} \de \mu_0(\bsigma)\nonumber\\
&= \beta^2\xi(q) + \beta^2/2 + \frac{1}{N} \log \mu_0(B_{q,\eps}(\bsigma_0)) + O(\eps)\nonumber\\
&= \beta^2\xi(q) + \beta^2/2 + h(q) + O(\eps) + o_N(1)\,,\label{eq:firstmoment}
\end{align}
where $h(q) =  \frac{1}{2}\log(1-q^2)$ in the spherical case $\Sigma_N = \Sph^{N-1}(\sqrt{N})$ and $h(q) = -\frac{1+q}{2}\log (1+q) - \frac{1-q}{2}\log (1-q)$ in the Ising case $\Sigma_N = \{-1,+1\}^N$. On the other hand the distribution of $Z_N$ under $\nu_{\pl}$ is invariant under the change of variable $g_{i_1,\cdots,i_k} \to g_{i_1,\cdots,i_k}\sigma_{0i_1}\cdots\sigma_{0i_k}$ so without loss of generality we may set $\bsigma_0 = \one$. Next, $\bG \mapsto \log Z_N$ is $\beta\sqrt{N}$-Lipschitz, so by Gaussian concentration we have
\begin{equation} \label{eq:concentr}
\P_{(\bG,\bsigma_0)\sim \nu_{\pl}}\Big(|\log Z_N - \bbE_{\plsmall}\log Z_N|\ge t N\Big) \le 2e^{-t^2N/(2\beta^2)}\,,~~~ t \ge 0\,.
\end{equation}
By contiguity of $\nu_{\pl}$ and $\nu_{\rd}$ at exponential scale (Lemma~\ref{lem:contig}) the event $|\log Z_N - \bbE_{\plsmall}\log Z_N|\ge t N$ has probability at most $e^{-cN}$, $c = c(\beta,t)>0$ under $(\bG,\bsigma_0) \sim \nu_{\rd}$. Using Jensen's inequality and letting $N \to \infty$ followed by $\eps \to 0$ yields the upper bound
\begin{equation}\label{FP:upbd}
F_{\beta}(q) - F(\beta)\le \beta^2\xi(q) + h(q) \, .
\end{equation}
A quick inspection of the first two derivatives at $q=0$ of the function on right-hand side reveals that if $\beta^2\xi''(0) < 1$ then there exists $\hat{q}>0$ such that $\beta^2\xi(q) + h(q) <0$ for $q \in (0,\hat{q})$. This proves the result.
\end{proof}

Now we show that the FP potential is globally dominated by the full free energy up to $\beta_c$: 
\begin{lemma}\label{lem:FPglobal}
For all $\beta < \beta_c$, $F_{\beta}(q) < F(\beta)$ for all $q \in (0,1)$.
\end{lemma}
\begin{proof}
Let $\beta < \beta_c$. We control $F_{\beta}(q)$ by introducing an external field; this idea will be used again to obtain an exact formula for the FP potential. In the present argument only the upper bound is needed. For $h \in \reals$,
\begin{align*}
\frac{1}{N}\log \int_{B_{q,\eps}(\bsigma_0)}  e^{\beta H_{N}(\bsigma)} \de \mu_0(\bsigma) &\le 
\frac{1}{N}\log \int  e^{\beta H_{N}(\bsigma) + h \langle \bsigma_0,\bsigma\rangle} \de \mu_0(\bsigma) - h q + \eps \,, 
\end{align*}
where we recall $B_{q,\eps}(\bsigma) = \{\bsigma' \in \Sigma_N : |\langle \bsigma , \bsigma'\rangle/N - q|\le \eps\}$. Using concentration of the random variable on the left-hand side above and contiguity at exponential scale between the random and planted distributions, exactly as done in Lemma~\ref{lem:FPbound} we obtain 
\begin{equation}\label{eq:FPbound2}
F_{\beta}(q) \le \beta^2 \xi(q) + \inf_{h} \big\{F(\beta,h) -hq\big\}\,,
\end{equation}
where $F(\beta,h) = \lim_{N\to \infty} \frac{1}{N}\E\log \int  e^{\beta H_{N}(\bsigma) + h \langle \bsigma_0,\bsigma\rangle} \de \mu_0(\bsigma)$ is the free energy of $H_N$ augmented with an external field, where without loss of generality we set $\bsigma_0 = \one$. Since this is given by the Parisi formula as stated in Proposition~\ref{prop:parisi-formula} we may use a replica-symmetric ansatz to upper bound $F_{\beta}(q)$: we consider the overlap distribution $\zeta = \delta_{q}$ and plug it in the Parisi formula in the Ising case. From Proposition~\ref{prop:parisi-formula} we obtain 
\begin{equation}\label{eq:FPbound3}
F_{\beta}(q) \le \beta^2 \xi(q) + \inf_{h}\Big\{\E \log \cosh\big(\sqrt{\beta^2 \xi'(q)}z + h\big) + \frac{\beta^2}{2} \big(\xi(1) - \xi(q) - (1-q)\xi'(q)\big) - hq \Big\}\, ,
\end{equation}
where $z \sim N(0,1)$. Next we choose $h = \beta^2\xi'(q)$. After simplification the above becomes 
\begin{equation}\label{eq:FPbound4}
F_{\beta}(q) \le \phi_{\RS}(q;\beta) := \E \log \cosh\big(\sqrt{\beta^2 \xi'(q)}z + \beta^2 \xi'(q)\big) + \frac{\beta^2}{2} \big(\xi(1) + \xi(q) - (1+q)\xi'(q)\big) \, .
\end{equation}
It remains to show that $\phi_{\RS}(q;\beta) < F(\beta) = \beta^2/2$ for all $\beta < \beta_c$ and $q\neq 0$. To this end we observe that the limit of the free energy $F_{\pl}(\beta) := (1/N) \E_{\plsmall} \log Z_N$ of $\mu_{\beta}$ under the planted model~\eqref{eq:pl} can be represented as a \emph{supremum} of $\phi_{\RS}(q;\beta)$ over $q$: 
\begin{equation}\label{eq:FPbound5}
F_{\pl}(\beta) = \sup_{q \ge 0} \phi_{\RS}(q;\beta) \,.
\end{equation}
This was proved for the pure $p$-spin case in~\cite[Theorem 1]{lesieur2017statistical}. Their proof technique is classical in spin glass theory and extends in a straightforward way to the general mixed case.
By contiguity at exponential scale we have $F_{\pl}(\beta) = F(\beta)$ for all $\beta\le\beta_c$.  
Next we show that the above supremum is uniquely achieved at $q=0$ for all $\beta<\beta_c$ thereby finishing the proof. Since $F_{\pl}(\beta) = \beta^2\xi(1)/2$ is differentiable for all $\beta<\beta_c$ (and recall $\xi(1)=1$), by the envelope theorem~\cite{milgrom2002envelope} we have
\begin{equation}\label{eq:rs2}
\beta\xi(1) = F_{\pl}'(\beta) = \beta \xi'(q)\big(\bbE \tanh^2\big(\sqrt{\beta^2 \xi'(q)}z + \beta^2 \xi'(q)\big) - q\big) + \beta(\xi(1)+\xi(q))\,,
\end{equation}
for all $\beta<\beta_c$ where $q$ is any maximizer of $\phi_{\RS}(\,\cdot\,;\beta)$. On the other hand any such maximizer must satisfy $\phi_{\RS}'(q;\beta)=0$. Indeed,
\begin{equation}\label{eq:rs3}
\phi_{\RS}'(q;\beta) = \frac{\beta^2}{2}\xi''(q)\big(\bbE \tanh^2\big(\sqrt{\beta^2 \xi'(q)}z + \beta^2 \xi'(q)\big) - q\big)\,,
\end{equation}
so $q=0$ is a stationary point and $\phi_{\RS}'(1;\beta)<0$ so $q=1$ is not a maximizer. We deduce from~\eqref{eq:rs2} and~\eqref{eq:rs3} that any maximizer must satisfy $\xi(q)=0$ when $\beta<\beta_c$, i.e., $q=0$. 
\end{proof}

\begin{remark}
Lemma~\ref{lem:FPbound} controls the FP potential \emph{locally} in the neighborhood of $q=0$ up to $\beta_\cont$, while Lemma~\ref{lem:FPglobal} provides \emph{global} control up to $\beta_c \le \beta_\cont$. The point is that in a situation where $\beta_c < \beta_\cont$, $q=0$ remains the unique local maximizer of the FP potential up to $\beta_c$, so any other maximizer which might appear at $\beta = \beta_c$ must emerge far away. This is exactly what happens when RSB appears discontinuously as shown in the next lemma.   
\end{remark}

\begin{lemma}\label{lem:discont_FP}
For $\beta \le \beta_c$, if $q \in \supp(\zeta_\beta)$ then $F_{\beta}(q) = F(\beta)$. In particular $F_{\beta}(0) = \beta^2/2$ for all $\beta \le \beta_c$, and if $\xi$ admits a discontinuous RSB transition then $F_{\beta_c}(\bar{q}) =  \beta_c^2/2$, with $\bar{q} = \lim_{\beta\uparrow \beta_c}
    q_{\max}(\zeta_{\beta})$.
 \end{lemma}

\begin{proof}
Suppose $q$ is such that $F_{\beta}(q) < F(\beta)$ and let $\eps>0$. Under the planted model $\nu_{\pl}$,
\begin{align}
\bbE_{\plsmall} \,\mu_{\beta}^{\otimes 2} \big(\{|\langle \bsigma_1, \bsigma_2\rangle/N - q| \le \eps \}\big) &= \bbE_{\plsmall} \,\mu_{\beta} \big(\{|\langle \bsigma_0, \bsigma_1\rangle/N - q| \le \eps \}\big) \nonumber\\
&= \bbE_{\plsmall}\left[\frac{\int_{B_{q,\eps}(\bsigma_0)} e^{\beta H_N(\bsigma)} \de \mu_0(\bsigma)}{\int e^{\beta H_N(\bsigma)} \de \mu_0(\bsigma)}\right]\,.\label{eq:ratio}
\end{align}
The first equality follows from the fact $(\bsigma_0,\bsigma_1) \stackrel{\de}{=} (\bsigma_1,\bsigma_2)$ since $\mu_\beta$ is the posterior distribution of $\bsigma$ given $\bG$ in the model~\eqref{eq:pl}; this is often referred to as Nishimori's identity, see e.g.,~\cite{lelarge2016fundamental}.   
Next, since the two events 
\begin{align*}
    &\big\{(1/N)\log \int_{B_{q,\eps}(\bsigma_0)} e^{\beta H_N(\bsigma)} \de \mu_0(\bsigma) \le F_{\beta}(q) + \eps/2\big\}\\
    &\mbox{and}~~~
\big\{(1/N)\log \int e^{\beta H_N(\bsigma)} \de \mu_0(\bsigma) \ge F(\beta) - \eps/2\big\}
\end{align*}
both have probability at least $1-e^{-c(\beta,\eps)N}$, 
$c(\beta,\eps)>0$, we obtain that the right-hand side of~\eqref{eq:ratio} is bounded by $e^{-c'N}$, $c'=c'(\beta,\eps)>0$. 
Then by contiguity at exponential scale, Lemma~\ref{lem:contig}, we obtain
\begin{equation}\label{eq:bound2}
\E \mu_{\beta}^{\otimes 2} \big(\{|\langle \bsigma_1, \bsigma_2\rangle/N - q| \le \eps \}\big) \le e^{-c''N}\,,~~~~c'' = c''(\eps,\beta)>0\,,
\end{equation}
and $\eps$ sufficiently small.
This implies $\zeta_{\beta}([q-\eps,q+\eps]) = 0$ via a standard approximation argument: we approximate the mixture $\xi$ with a \emph{generic} mixture containing a non-zero monomial of every degree, i.e., $\gamma_p \neq 0$ for all $p \ge 3$, and let $\tilde{\zeta}_{\beta}$ be the corresponding Parisi measure.     
By approximating the indicator function appearing in~\eqref{eq:bound2} with polynomials and using Proposition~\ref{prop:free-energy-deriv}, we obtain  
\[\tilde{\zeta}_{\beta}([q-\eps,q+\eps]) = 0\,,\]
by taking $N\to \infty$. Finally since the Parisi measure is continuous in the coefficients of the mixture $(\gamma_p)$ this implies $\zeta_{\beta}([q-\eps,q+\eps]) = 0$ for all $\eps$ sufficiently small, and therefore $q \notin \supp(\zeta_{\beta})$. 
\end{proof}

Now we are ready to prove Theorem~\ref{thm:main}:
\begin{proof}[Proof of Theorem~\ref{thm:main}]
 On the one hand by Lemma~\ref{lem:FPbound} we have $F_{\beta}(q) < F(\beta)$ for all $\beta < \beta_c$ (since $\beta_{c} < \beta_{\cont}$ by assumption) and $q$ in a small neighborhood of zero: there exists $\underline{q}, \eps>0$ depending only on $\beta_c$ such that $F_{\beta}(\underline{q})\le F_{\beta}(0)-\eps$ for all $\beta \le \beta_c$. Furthermore, since $\xi$ admits a discontinuous RSB transition, $\bar{q} = \lim_{\beta\downarrow \beta_c}
    q_{\max}(\zeta_{\beta})>0$, and we can take $\eps>0$ small enough so that $\underline{q} <  \bar{q}$.

On the other hand,
as a consequence of Lemma~\ref{lem:discont_FP} and continuity of $\beta \mapsto F_{\beta}(q) - \beta^2/2$, there exists $\delta>0$ such that 
\[\beta^2/2 - \eps/2 \le F_{\beta}(\bar{q}) \le \beta^2/2 \, ,~~~ \forall \beta \in [\beta_c-\delta,\beta_c]\,.\]
It follows that $F_{\beta}(\underline{q}) < F_{\beta}(\bar{q})$.
\end{proof}

\section{The Franz--Parisi potential}
\label{sec:FP}

Here we compute the Franz--Parisi potential $F_{\beta}(q)$ for $\beta\leq\beta_c$ in the Ising case. The spherical case was previously treated in \cite{alaoui2023shattering} via a technique leveraging spherical symmetry. We provide here a derivation in the Ising case via a duality-based route.

Here $\mu_0$ is the uniform distribution on $\Sigma_N=\{-1,+1\}^N$. We proceed by duality with a model with non-zero external field. 
For $h \in \reals$ we consider the free energy with external field $h \mathbf{1}$:  
 \begin{align}
F(\beta,h) := \lim_{N\to \infty} \frac{1}{N} \E \int \sum_{\bsigma \in \Sigma_N} e^{\beta H_N(\bsigma) + h\langle \mathbf{1} , \bsigma \rangle} \de\mu_0(\bsig).
\end{align}

\begin{proposition}\label{prop:duality}
$F(\beta, h)$ is continuously differentiable in $h$.
Consequently for any $q \in (-1,1)$, the equation 
\begin{equation}\label{eq:stateq}
    \frac{\de}{\de h} F(\beta, h) = q\, .
\end{equation} 
has a solution $h = \sup\{h: \frac{\de}{\de h} F(\beta, h) = q\}$.
Hence the map $q \mapsto h(q)$ is strictly increasing.
Finally $\frac{\de}{\de h} F(\beta, 0)=0$.
\end{proposition}

\begin{proof}
Proposition~\ref{prop:free-energy-deriv} and the envelope theorem \cite{milgrom2002envelope} implies that $h\mapsto F(\beta,h)$ is continuously differentiable with 
\[
\frac{\partial F(\beta,h)}{\partial h}
=
\partial_x \Phi_{\zeta_{\beta}(\xi,h)}(0,h).
\]
Furthermore, $F(\beta,h)$ is convex in $h$ because $\log \int e^{\beta H_N(\bsigma) + h\langle \mathbf{1} , \bsigma \rangle} \de\mu_0(\bsig)$ is.
Further since $\frac{\de}{\de h} F(\beta,h)\to \pm 1$ for $h \to \pm \infty$, there is a solution $h$ to Eq.~\eqref{eq:stateq} for any $q \in (-1,1)$. 
The last assertion follows since $F(\beta,h)=F(\beta,-h)$ by symmetry.
\end{proof}

\begin{proposition}\label{prop:fp_ising}
For $\beta < \beta_c$ and $q \in (-1,1)$, the function $F_{\beta}$ is 
given by:
\[
F_{\beta}(q) = F(\beta,h(q)) - h(q)q + \beta^2 \xi(q) \, .
\]
\end{proposition}

\begin{proof}
We will show that for all $h \in \reals$ and $\beta < \beta_c$, 
 \begin{equation}\label{eq:legendre}
 F(\beta,h) = \sup_{q \in [-1,1]} \Big\{F_\beta(q) - \beta^2 \xi(q) + hq\Big\}\,.
 \end{equation}

Momentarily assuming this, by continuity in $q$ the above supremum is achieved. Since $F(\beta,h)$ is convex and continuously differentiable in $h$, we obtain by the envelope theorem that $\frac{\de}{\de h}F(\beta,h) = q^*$ for any $q^* \in \argmax_{q \in [-1,1]} \{F_{\beta}(q) - \beta^2\xi(q) + hq\}$. Therefore the maximum is uniquely attained at $q(h) = \frac{\de}{\de h}F(\beta,h) \in (-1,1)$. By Proposition~\ref{prop:duality} this function is surjective with right-continuous inverse $h$, and we obtain
\[F_{\beta}(q) = F(\beta,h(q)) +\beta^2 \xi(q) - h(q)q\,.\]  

Now we show Eq.~\eqref{eq:legendre}.
For $\bsigma_0 \in \Sigma_N$ we recall the notation $B_{q,\eps}(\bsigma_0) = \{\bsigma \in \Sigma_N : |\langle \bsigma_0 , \bsigma\rangle/N - q|\le \eps\}$. 
First, using the planted model we have
\[F_{\beta}(q) = \beta^2\xi(q) + \lim_{\eps \to 0}\lim_{N \to \infty}\frac{1}{N} \E\log \int_{B_{q,\eps}(\bsigma_0)}  e^{\beta H_N(\bsigma)} \de\mu_0( \bsigma)\,,\]
where $\bsigma_0 \sim \mu_0$. By symmetry we may assume that $\bsigma_0 = \mathbf{1}$.   
For the lower bound, 
\begin{align*}
    \frac{1}{N} \E \log \int  e^{\beta H_N(\bsigma) + h\langle \bsigma_0 , \bsigma \rangle} \de\mu_0( \bsigma) &\ge \frac{1}{N} \E \max_{q \in \eps \Z \cap [-1,1]} \log \int_{B_{q,\eps}(\bsigma_0)}  e^{\beta H_N(\bsigma) + h\langle \bsigma_0 , \bsigma \rangle} \de\mu_0( \bsigma)\\ 
   &\ge  \max_{q \in \eps \Z \cap [-1,1]} \frac{1}{N} \E\log \int_{B_{q,\eps}(\bsigma_0)}  e^{\beta H_N(\bsigma)} \de\mu_0( \bsigma) + h(q-\eps) \, .
\end{align*}
Taking $N \to \infty$ then $\eps \to 0$ we obtain 
\[F(\beta,h) \ge \sup_{q \in [-1,1]} \Big\{F_{\beta}(q) - \beta^2\xi(q) + hq\Big\}\,.\]
As for the upper bound, we have 
\begin{align*}
   F(\beta, h)  &=  \frac{1}{N} \E \log \Big(\sum_{q \in \eps \Z \cap [-1,1]} \int_{B_{q,\eps}(\bsigma_0)} e^{\beta H_N(\bsigma) + h\langle \bsigma_0 , \bsigma \rangle} \de\mu_0( \bsigma)\Big)\\ 
   &\le \lim_{N\to \infty} \frac{1}{N} \E \max_{q \in \eps \Z \cap [-1,1]} \log \int_{B_{q,\eps}(\bsigma_0)} e^{\beta H_N(\bsigma) + h\langle \bsigma_0 , \bsigma \rangle} \de\mu_0( \bsigma) + \frac{\log(2/\eps)}{N} \\
   &\le  \frac{1}{N} \E\max_{q \in \eps \Z \cap [-1,1]} \log \int_{B_{q,\eps}(\bsigma_0)} e^{\beta H_N(\bsigma)} \de\mu_0( \bsigma) + h(q +\eps) + \frac{\log(2/\eps)}{N}\\
   &\le   \max_{q \in \eps \Z \cap [-1,1]}\frac{1}{N} \E \log \int_{B_{q,\eps}(\bsigma_0)} e^{\beta H_N(\bsigma)} \de\mu_0( \bsigma) + h(q +\eps) + \frac{\log(2/\eps)}{N}+o_N(1)\,.
\end{align*}
The last line follows by a standard application of the tail bound $\P(\log Z_N - \E \log Z_N \ge t) \le e^{-t^2/(2\beta^2 N)}$ where $Z_N = \int_{B_{q,\eps}(\bsigma_0)} e^{H_N(\bsigma)} \de \mu_0(\bsigma)$ (c.f.~\eqref{eq:concentr}). This allows us to swap the order of the maximum and the expectation via a union bound over the finite set $\eps\Z \cap [-1,1]$. Letting $N \to \infty$ then $\eps \to 0$ we obtain the converse bound, therefore establishing~\eqref{eq:legendre}.
\end{proof}

\section{An alternative shattering criterion on the sphere}
\label{sec:spherical}

On the sphere, an alternate criterion for shattering at inverse temperature $\beta$ is that there exists $q\in (0,1)$ with 
\begin{equation}
\label{eq:dynamic-condition}
    \beta^2 \xi'(q)>\frac{q}{1-q} \,. 
\end{equation}
Here we show a similar result for this alternative criterion.
Here, relevant thresholds become more explicit. By~\cite[Proposition 2.3]{talagrand2006spherical}, the static replica symmetry breaking inverse temperature $\beta_c$ is given by the formula
\begin{equation}
\label{eq:beta-c-sphere}
\beta_c
=
\sqrt{\inf_{q \in (0, 1)} \frac{-\log(1-q)-q}{\xi(q)}}\,.
\end{equation}
On the other hand we define the dynamical inverse temperature $\beta_d$ by
\begin{align*}
\beta_d
&=
\min(\bar\beta_d,\beta_c)\,,
\\
\bar\beta_d
&=
\sqrt{\inf_{q \in (0, 1)} \frac{q}{(1-q)\xi'(q)}}\,, 
\end{align*}
which is the onset of \eqref{eq:dynamic-condition}.

We define separate critical temperatures for continuous and discontinuous breaking of replica-symmetry.
Recall $\beta_{\cont}= (\xi''(0))^{-1/2}$; this is the threshold at which $\delta_0$ becomes unstable to local perturbations, and let $\beta_{\dis}^2$ be the smallest value of 
\begin{equation}\label{eq:kappa}
\kappa(q)
=
\frac{-\log(1-q)-q}{\xi(q)}
\end{equation}
at any of its nonzero critical points.
This is the threshold for an atom to emerge away from $0$. Indeed, Proposition 2.1 of~\cite{talagrand2006spherical} shows that the support of the Parisi measure $\zeta_{\beta}$ is contained in the set of maximizers of a function $f: [0,1]\to \R$ which for $\beta \le \beta_c$ takes the form   
\begin{equation}\label{eq:f}
f(q) = \beta^2 \xi(q) + \log(1-q)+q\,.
\end{equation}
For $\beta \le \beta_c$, since $\zeta_{\beta} = \delta_0$, this function is nonpositive with $q=0$ as a maximizer. Therefore if $\xi$ admits a discontinuous RSB transition, there must exist a second maximizer $q_c \neq 0$ at $\beta = \beta_c$, i.e., $f(q_c) = 0$. This point is a nonzero  minimizer of $\kappa$ which is also a critical point: $\kappa'(q_c)=0$.      
It is then immediate from \eqref{eq:beta-c-sphere} that 
\[
\beta_c
=
\min(\beta_{\cont},\beta_{\dis})\,.
\]

\begin{proposition}
    Suppose $\xi$ admits a discontinuous RSB transition on the sphere and $\beta_c < \beta_\cont$.
    Then there exists $q$ with $\beta_c^2 \xi'(q)>\frac{q}{1-q}$.
\end{proposition}

\begin{proof}
    By the above discussion, a discontinuous RSB transition means there exists $q_c\in (0,1)$ with
    \[
    f(q_c) = \beta_c^2 \xi(q_c)+q_c+\log(1-q_c)=0\,.
    \]
    By the mean-value theorem, there exists $q_d\in [0,q_c]$ with $\xi'(q_d)\geq \frac{q_d}{1-q_d}$.
    Strict inequality holds for some $q_d$ unless equality holds on all of $(0,q_c)$; this cannot be since $\xi(1)$ is finite.
\end{proof}

We also prove that shattering can occur together with a continuous RSB transition, essentially since the converse of the mean-value theorem is false.

\begin{proposition}
For any $p\geq 3$, there is $\gamma_p$ such that $\xi(q)=\frac{q^2}{2}+\gamma_p^2 q^p$ admits a continuous RSB transition, $\beta_c=1$, but $\beta_d<1$.
\end{proposition}

\begin{proof}
By Talagrand's characterization of the support of $\zeta_\beta$, see Eq~\eqref{eq:f}, $\xi$ admits a continuous RSB transition if and only if $f(q)<0$ for all $q \neq 0$ at $\beta = \beta_{\cont} = 1$. Equivalently, 
\[
\gamma_p^2 
\leq 
\min_{q\in (0,1)}
\frac{\sum_{k\geq 3} q^k /k}{q^p}\, .
\]
On the other hand, there exists $q$ with $\xi'(q)>\frac{q}{1-q}$ if 
\[
\gamma_p^2 
>
\min_{q\in (0,1)}
\frac{q^2}{pq^{p-1}(1-q)}\,.
\]
It therefore remains to prove that
\begin{equation}
\label{eq:remains-to-prove}
\min_{q\in (0,1)}
\frac{q^2}{pq^{p-1}(1-q)}
<
\min_{q\in (0,1)}
\frac{\sum_{k\geq 3} q^k /k}{q^p}\,.
\end{equation}
Indeed,
\[
\frac{\sum_{k\geq 3} q^k /k}{q^p}
=
\frac{\int_0^q \frac{t^2}{1-t}\de t}{\int_0^q pt^{p-1}\de t}
\,.
\]
This implies \eqref{eq:remains-to-prove} with non-strict inequality.
Strictness follows because the function $\frac{t^2}{pt^{p-1}(1-t)}$ is non-constant on any open set. 
\end{proof}

We end this section with a lemma on the genericity of the condition $\beta_c < \beta_{\cont}$. We emphasize the dependence on $\xi$ by writing $\beta_{\cont}(\xi)$ and $\beta_c(\xi)$.

\begin{proposition}\label{prop:beta-cont-bound}
We have $\beta_c(\xi) \le \beta_{\cont}(\xi)$ for any $\xi$ in both the Ising and spherical models.
\end{proposition}

\begin{proof}
    In the spherical case, we use the explicit characterization~\eqref{eq:beta-c-sphere} of $\beta_c$.
    We assume that $\xi''(0)>0$ as otherwise the statement is vacuous. 
    The limit as $q \to 0^+$ of the function $\kappa$ defined in Eq.~\eqref{eq:kappa} is $1/\xi''(0)$. This yields $\beta_c(\xi) \le \beta_{\cont}(\xi)$.
    In the Ising case, one can proceed similarly using the stationarity condition from e.g. \cite[Theorem 2]{chen2017variational}.

    Alternatively, in both cases one can deduce the bound as follows. 
    It is known in the pure $2$-spin case that $\beta_c=\beta_{\cont}$ is the RSB transition.
    By the subadditivity result \cite[Corollary 2.1]{sellke2022free}, decreasing the coefficients $\gamma_p$ of any RS spin glass model without external field yields another RS spin glass (since the RS or annealed free energy is always an upper bound for $F_{\beta}$, and is exactly additive).
    Applying this coefficient reduction to turn $\xi$ into a $2$-spin model gives the desired result.
    \end{proof}

\begin{lemma}\label{lem:beta_order}  
If a spherical spin glass with mixture $\xi$ admits a discontinuous RSB transition then there exists $\eps_0>0$ depending on $\xi$ such that any perturbation $\tilde{\xi}$ of $\xi$ of the form $\tilde{\xi}(x) = \xi(x) + \eps( x^p - x^2)$ with $0<\eps \le \eps_0$ and $p \ge 3$ also admits a discontinuous RSB transition and furthermore satisfies $\beta_c(\tilde{\xi}) < \beta_\cont(\tilde{\xi})$.
\end{lemma}

\begin{proof}
Suppose $\xi$ admits a discontinuous RSB transition and let $\tilde{\xi}(x) = \xi(x) + \eps (x^p - x^2)$ for $\eps>0$ and $p \ge 3$. As argued above, a discontinuous RSB implies that the function $\kappa$ is minimized at a point $q_c\in (0,1)$.
 We consider the following function of $\eps$ given by
\begin{equation*}
\varphi(\eps) := \frac{-\log(1-q)-q}{\xi(q)+\eps (q^p - q^2)} - \frac{1}{\xi''(0)-2\eps}\,,~~\mbox{with }~ q = q_c\,. 
\end{equation*}
Evaluating its derivative at $\eps=0$ we obtain
\begin{align*}
\varphi'(0) &= \frac{(-\log(1-q)-q)(q^2 - q^p)}{\xi(q)^2} - \frac{2}{\xi''(0)^2}\\
&\le \frac{q^2 - q^p}{\xi(q)\xi''(0)} - \frac{2}{\xi''(0)^2}\\
&= \frac{1}{\xi(q)\xi''(0)^2}\big[(q^2 - q^p)\xi''(0) - 2\xi(q)\big]\,, 
\end{align*}
where the inequality in the second line uses the fact $\kappa(q) \le 1/\xi''(0)$. 
Since $\xi(0) = \xi'(0)=0$ and all the mixture coefficients are nonnegative $\gamma_p \ge 0$ we have $\xi(x) - x^2\xi''(0)/2 \ge 0$ for all $x \ge 0$. Since $q>0$, the above derivative is strictly negative. Therefore, since $\varphi(\eps) = \kappa(q) - 1/\xi''(0) \le 0$, for all $\eps$ small enough we have $\varphi(\eps)<0$ which implies $\beta_c(\tilde{\xi}) < \beta_{\cont}(\tilde{\xi})$.  
\end{proof}

\small

\bibliographystyle{alpha}

\begin{thebibliography}{FPRT{\etalchar{+}}13}

\bibitem[AC15]{auffinger2015parisi}
Antonio Auffinger and Wei-Kuo Chen.
\newblock {The Parisi formula has a unique minimizer}.
\newblock {\em Comm. Math. Phys.}, 335(3):1429--1444, 2015.

\bibitem[Ala24]{alaoui2024near}
Ahmed~El Alaoui.
\newblock {Near-optimal shattering in the Ising pure $p$-spin and rarity of
  solutions returned by stable algorithms}.
\newblock {\em arXiv preprint arXiv:2412.03511}, 2024.

\bibitem[AMS25a]{alaoui2022sampling}
Ahmed~El Alaoui, Andrea Montanari, and Mark Sellke.
\newblock Sampling from mean-field gibbs measures via diffusion processes.
\newblock {\em Prob. Math. Phys., to appear}, 2025.

\bibitem[AMS25b]{alaoui2023shattering}
Ahmed~El Alaoui, Andrea Montanari, and Mark Sellke.
\newblock Shattering in pure spherical spin glasses.
\newblock {\em Communications in Mathematical Physics}, 406(5):1--36, 2025.

\bibitem[BLM13]{boucheron2013concentration}
St{\'e}phane Boucheron, G{\'a}bor Lugosi, and Pascal Massart.
\newblock {\em Concentration inequalities: A nonasymptotic theory of
  independence}.
\newblock Oxford university press, 2013.

\bibitem[Che13]{chen2013aizenman}
Wei-Kuo Chen.
\newblock {The Aizenman-Sims-Starr scheme and Parisi Formula for Mixed $p$-spin
  Spherical Models}.
\newblock {\em Electronic Journal of Probability}, 18:1--14, 2013.

\bibitem[Che17]{chen2017variational}
Wei-Kuo Chen.
\newblock {Variational representations for the Parisi functional and the
  two-dimensional Guerra--Talagrand bound}.
\newblock {\em Ann. Probab.}, 45(6A):3929--3966, 2017.

\bibitem[CHS93]{crisanti1993spherical}
Andrea Crisanti, Heinz Horner, and H~J Sommers.
\newblock The spherical p-spin interaction spin-glass model: the dynamics.
\newblock {\em Zeitschrift f{\"u}r Physik B Condensed Matter}, 92:257--271,
  1993.

\bibitem[DS24]{dembo2024disordered}
Amir Dembo and Eliran Subag.
\newblock {Disordered Gibbs measures and Gaussian conditioning}.
\newblock {\em arXiv preprint arXiv:2409.19453}, 2024.

\bibitem[DS25]{dembo2025dynamics}
Amir Dembo and Eliran Subag.
\newblock {Dynamics for spherical spin glasses: Gibbs distributed initial
  conditions}.
\newblock {\em arXiv preprint arXiv:2503.23342}, 2025.

\bibitem[EKZ21]{eldan2021spectral}
Ronen Eldan, Frederic Koehler, and Ofer Zeitouni.
\newblock {A spectral condition for spectral gap: fast mixing in
  high-temperature Ising models}.
\newblock {\em Probability Theory and Related Fields}, pages 1--17, 2021.

\bibitem[FPRT{\etalchar{+}}13]{franz2013note}
Silvio Franz, Giorgio Parisi, Federico Ricci-Tersenghi, Tommaso Rizzo, and
  Pierfrancesco Urbani.
\newblock A note on weakly discontinuous dynamical transitions.
\newblock {\em The Journal of Chemical Physics}, 138(6), 2013.

\bibitem[GJ19]{gheissari2019spectral}
Reza Gheissari and Aukosh Jagannath.
\newblock On the spectral gap of spherical spin glass dynamics.
\newblock In {\em Annales de l'Institut Henri Poincar{\'e}, Probabilit{\'e}s et
  Statistiques}, volume~55, pages 756--776. Institut Henri Poincar{\'e}, 2019.

\bibitem[KT23]{kirkpatrick2014universal}
Theodore~R Kirkpatrick and Dave Thirumalai.
\newblock Universal aspects of the structural glass transition from density
  functional theory.
\newblock In {\em Spin Glass Theory and Far Beyond: Replica Symmetry Breaking
  After 40 Years}, pages 115--133. World Scientific, 2023.

\bibitem[LM19]{lelarge2016fundamental}
Marc Lelarge and L{\'e}o Miolane.
\newblock Fundamental limits of symmetric low-rank matrix estimation.
\newblock {\em Probability Theory and Related Fields}, 173:859--929, 2019.

\bibitem[LML{\etalchar{+}}17]{lesieur2017statistical}
Thibault Lesieur, L{\'e}o Miolane, Marc Lelarge, Florent Krzakala, and Lenka
  Zdeborov{\'a}.
\newblock Statistical and computational phase transitions in spiked tensor
  estimation.
\newblock In {\em 2017 IEEE International Symposium on Information Theory
  (ISIT)}, pages 511--515. IEEE, 2017.

\bibitem[MP96]{mezard1996tentative}
Marc M{\'e}zard and Giorgio Parisi.
\newblock A tentative replica study of the glass transition.
\newblock {\em Journal of Physics A: Mathematical and General}, 29(20):6515,
  1996.

\bibitem[MS02]{milgrom2002envelope}
Paul Milgrom and Ilya Segal.
\newblock Envelope theorems for arbitrary choice sets.
\newblock {\em Econometrica}, 70(2):583--601, 2002.

\bibitem[Pan08]{panchenko2008differentiability}
Dmitry Panchenko.
\newblock {On differentiability of the Parisi formula}.
\newblock {\em Electronic Communications in Probability}, 13:241--247, 2008.

\bibitem[Pan13]{panchenko2013parisi}
Dmitry Panchenko.
\newblock {The Parisi ultrametricity conjecture}.
\newblock {\em Annals of Mathematics}, pages 383--393, 2013.

\bibitem[Par79]{parisi1979infinite}
Giorgio Parisi.
\newblock Infinite number of order parameters for spin-glasses.
\newblock {\em Physical Review Letters}, 43(23):1754, 1979.

\bibitem[Sel23]{sellke2022free}
Mark Sellke.
\newblock {Free energy subadditivity for symmetric random Hamiltonians}.
\newblock {\em Journal of Mathematical Physics}, 64(4), 2023.

\bibitem[Tal06a]{talagrand2006spherical}
Michel Talagrand.
\newblock Free energy of the spherical mean field model.
\newblock {\em Probab. Theory and Related Fields}, 134(3):339--382, 2006.

\bibitem[Tal06b]{talagrand2006parisi-b}
Michel Talagrand.
\newblock Parisi measures.
\newblock {\em Journal of Functional Analysis}, 231(2):269--286, 2006.

\bibitem[Tal06c]{talagrand2006parisi}
Michel Talagrand.
\newblock {The Parisi formula}.
\newblock {\em Annals of Mathematics}, pages 221--263, 2006.

\end{thebibliography}
\newcommand{\etalchar}[1]{$^{#1}$}

\end{document}